\documentclass[12pt]{amsart}
\usepackage{amsmath,amssymb}
\usepackage{amsthm}
\usepackage{datetime}
\newtheorem{theorem}{Theorem}[section]

\newtheorem{lemma}[theorem]{Lemma}
\theoremstyle{definition}

\newtheorem{remark}[theorem]{Remark}
\newtheorem{remarks}[theorem]{Remarks}

\makeatletter 
\@addtoreset{equation}{section} 
\makeatother

\begin{document}

\title[On a Third Order Newton's Approximation Method]{On the Dynamics of a 
Third Order Newton's Approximation Method} 

\author[A. Gheondea]{Aurelian Gheondea}
\address{Department of Mathematics, Bilkent University, 06800 Bilkent, Ankara, 
Turkey, \emph{and} Institutul de Matematic\u a al Academiei Rom\^ane, C.P.\ 
1-764, 014700 Bucure\c sti, Rom\^ania} 
\email{aurelian@fen.bilkent.edu.tr \textrm{and} A.Gheondea@imar.ro}

\author[M.E. \c Samc\i]{Mehmet Emre \c Samc\i}
\address{Department of Mathematics, Bilkent University, 06800 Bilkent, Ankara, 
Turkey, \emph{current address} College of Administrative Sciences 
and Economics, Ko\c c University, Rumel{\protect\.{\i}}fener{\protect\.{\i}} K\"oy\"u, 
34450 Sar\i yer/\.Istanbul, Turkey}
\email{mehmet.emre.samci@gmail.com}

\begin{abstract} We provide an answer to a question raised by
S. Amat, S. Busquier, S. Plaza 
%in \textit{J. Math. Anal. Appl.}, 366(2010), 24--32, 
on the qualitative analysis of
the dynamics of the third order Newton type approximation function $M_f$, by proving that
for functions $f$ twice continuously differentiable and such that both $f$ and
its derivative do not have multiple roots, 
with at least four roots and infinite
limits of opposite signs at $\pm\infty$, $M_f$ has
periodic points of any prime period and that the set of points $a$ at which the
approximation sequence $(M_f^n(a))_{n\in\mathbb{N}}$ does not converge is
uncountable. In addition, we observe that in their Scaling Theorem 
analyticity can be replaced with differentiability.
\end{abstract}


\date{\today,\ \currenttime}

\maketitle

\section{Introduction} 

The classical Newton's Approximation Function $N_f(x)=x-f(x)/f^\prime(x)$, for
numerical approximation of roots of (nonlinear) functions $f$, under
certain conditions of smoothness and distribution of roots and critical
points, has a second order speed of convergence and, until now, it is considered
as one of the most useful and reliable iterative method of this kind. 
S.~Amat, S.~Busquier, and S.~Plaza, in
\cite{AmatBusquierPlaza}, modified it to
a third order approximation function $M_f(x)=N_f(x)-f(N_f(x))/f^\prime(x)$
that is free of second derivatives and shows a
remarkable robustness when compared to other methods. 
On the other hand, it is known that the classical
Newton's Approximation Function $N_f$, 
when considered  as a discrete dynamical system, shows chaotic
behaviour, at least from two possible acceptions of the concept of chaos: 
first in the sense of T-Y.~Li and J.A.~Yorke \cite{LiYorke}, that is, 
existence of periodic points of any prime period,
and second in the sense of R.~Bowen \cite{Bowen}, that is, a strictly positive
topological entropy, as shown by M.~Hurley and C.~Martin 
\cite{HurleyMartin}, see also D.G.~Saari and J.B.~Urenko \cite{SaariUrenko} for
similar investigations. 

In \cite{AmatBusquierPlaza}, the chaotic behaviour of $M_f$ 
was numerically pointed out for polynomials of order less or equal than $3$
 by using a bifurcation diagram similar to that of the logistic
map. They left open the question of performing a qualitative analysis on the discrete dynamical 
system associated to $M_f$ in order to mathematically prove its chaotic behaviour. 
The main result of this article is
Theorem~\ref{t:main} that shows that for functions $f$ of Newton type, see
Section~\ref{s:dtontm} for definition, with at least four roots and infinite
limits of opposite signs at $\pm\infty$, $M_f$ has
periodic points of any prime period and the set of points $a$ at which the
approximation sequence $(M_f^n(a))_{n\in\mathbb{N}}$ does not converge is
uncountable. For example, when considering polynomials,
this result applies to all odd degree polynomials with a certain distribution 
of real roots. In view of Lemma~\ref{l:jkn},
Theorem~\ref{t:main} might be extended to other classes of
Newton's functions $f$ having at least four roots and for which $M_f$ has at
least two bands that cover the whole real line $\mathbb{R}$, that is, there are two pairs
of disjoint open intervals, formed by consecutive critical points of $f$, 
that are mapped by $M_f$ onto the whole $\mathbb{R}$, see \cite{HurleyMartin} for 
precise terminology.

In addition, in Theorem~\ref{t:scaling}
we observe that in the Scaling Theorem from \cite{AmatBusquierPlaza},
which says that the dynamics of $M_f$ is stable under affine conjugacy, as
well as in its damped version as in \cite{AmatBusquierMagrenan}, 
analyticity of $f$ can be replaced by its differentiability.
The Scaling Theorem is essential for the analysis 
performed in
\cite{AmatBusquierPlaza} because it reduces the study
of the dynamics of $M_f$ for a general class of functions $f$ 
to the study of the dynamics of $M_f$ for a considerably 
smaller class of simpler functions. 
For example, in order to understand the dynamics of $M_f$ for
quadratic polynomials $f$, it
suffices to study only the dynamics of $M_f$ for the quadratic polynomials 
$x^2,\ x^2+1,\ x^2-1$, while for cubic polynomials $f$, it suffices 
to study only the dynamics of $M_f$ for the cubic polynomials
$x^3,\ x^3+1,\ x^3-1,\ x^3+\gamma x+1$, with $\gamma \in\mathbb{R}$.

We might also add results on the lower estimation of the Bowen's
topological entropy for $M_f$ but, 
once Theorem~\ref{t:main} is obtained, these
follow in an almost identical fashion as in \cite{HurleyMartin}, when
considered the two kind of bands induced by the critical points of $f$ for $M_f$.
Also, we observe that, for the class of functions that make the assumptions of
Theorem~\ref{t:main}, the damping with a parameter $\lambda$, cf.\
\cite{AmatBusquierMagrenan}, does
not change either the chaotic behaviour expressed as the existence of periodic
points of any prime period or the uncountability of the set of all real
numbers $a$ for which the iterative sequences
$(M_{\lambda,f}^n(a))_{n=1}^\infty$ diverges.

\section{Preliminaries}\label{s:p}

In this section we collect results related to periodic points of continuous
real functions. We start by proving a sequence of lemmas that are 
essentially contained in \cite{HurleyMartin} and more or less
implicit/explicit in the works of N.A.~Sharkovsky \cite{Sharkovsky} and T.-Y.~Li and 
J.A.~Yorke \cite{LiYorke}. 

The first lemma is a well known fixed point result and a direct consequence of the 
Intermediate Value Property for continuous functions.

\begin{lemma} \label{l:kk}
If $I$ and $J$ are compact intervals and $f\colon I \rightarrow J$ 
is a continuous function with $f(I)\supseteq I$, then $f$ has a fixed point. 
\end{lemma}

The next lemma is important for understanding the dynamics of continuous
real functions, e.g.\ see Lemma 0 in \cite{LiYorke}. We provide a proof for
consistency. 

\begin{lemma} \label{l:jk}
Let $J$ and $K$ be nonempty compact intervals
and let $f \colon J \rightarrow \mathbb{R}$ be a continuous
function  such that $f(J) \supseteq K$. 
Then, there exists a nonempty compact interval $ L \subseteq J $ such
that $f(L) = K$. 
\end{lemma}

\begin{proof} Let $K=[a,b]$. If $a=b$ then we apply the Intermediate Value
  Theorem and get $c\in J$ such that $f(c)=a$ and let $L=[c,c]$, so let us
  assume that $a<b$.
Since the set
$\{ x \in J \mid f(x) = a \}$ is compact and nonempty, 
there is a greatest element $c$ in this set. If $f(x) = b$ for some 
$x \geq c$ with $x\subseteq J$, then $x>c$
and letting $d$ be the least of them, by the Intermediate Value Theorem 
$f([c,d]) = K$ and we let $L=[c,d]$. 
Otherwise, $f(x) = b$ for some $x < c$ and let
$c'$ be the largest of them. Then let $d'$ be the smallest of the set of all 
$x>c'$ with $f(x) = a$
so that $[c',d']$ is an interval in $J$. Then $f([c',d'])= K$ and we let
$L=[c',d']$.  
\end{proof}

The main technical fact we use 
is a refinement of Lemma 2.2 in \cite{HurleyMartin}, 
implicit in the proof of Sharkovsky's Theorem
\cite{Sharkovsky}, see also \cite{LiYorke}. Recall
that, a point $a\in M$ is called a \emph{periodic point} 
for a function
$g\colon M\rightarrow M$ 
if there exists $n\in\mathbb{N}$ such that $g^n(a)=a$, and the least
$n\in\mathbb{N}$ with this property is called its \emph{prime period}.

\begin{lemma} \label{l:jkn}
Let $g\colon\mathbb{R} \rightarrow \mathbb{R}$ be a function and let
$I_{1},I_{2},\ldots,I_{k}$ be compact, disjoint and nondegenerate
intervals, with $k\geq 2$, such that, for all $ m \in\{ 1,2,\ldots,k\}$,
$g$ is continuous on $I_m$ and
\begin{equation*}
g(I_{m}) \supseteq\displaystyle\bigcup\limits_{j=1}^{^k}I_{j}.
\end{equation*}
Then:

\emph{(a)} For each $n\in\mathbb{N}$, $g$ has at
  least $k(k-1)^{n-1}$ periodic points 
of prime period $n$, in particular, $g$ has periodic points of any prime period.

\emph{(b)} The set of all real numbers $a$ for which the orbit
$(g^n(a))_{n\in\mathbb{N}}$ makes a sequence that does not converge is uncountable.
\end{lemma}

\begin{proof} (a) For any $n\in\mathbb{N}$,
take any sequence $(j_{i})^{n}_{i=1}$ of length $n$ with 
$j_{i} \in \{1, 2,\ldots, k \}$
and let $  j_{n+1} = j_{1}$. Considering the sequence of intervals
$(I_{j_{i}})^{n+1}_{i=1}$, by assumption, 
\begin{equation*}
 g(I_{j_{1}}) \supseteq \displaystyle\bigcup\limits_{j=1}^k J_j\supseteq\displaystyle\bigcup\limits_{i=1}^{n}I_{j_{i}} \supseteq
 I_{j_{2}}.\end{equation*} 
By Lemma~\ref{l:jk} there exists a compact interval
$A_{j_{1}} \subseteq I_{j_{1}}$ such that $g(A_{j_{1}})=
I_{j_{2}}$. If $n=1$, we observe that, since $g(A_{j_{1}})=
I_{j_{2}}=I_{j_{1}}$, by Lemma~\ref{l:kk} it follows that $g$ has a fixed point in the compact interval
$A_{j_1}$.

If $n\geq 2$, then 
\begin{equation*}g^2(A_{j_{1}}) = g(I_{j_{2}}) \supseteq\displaystyle\bigcup\limits_{j=1}^k J_j
\supseteq\displaystyle\bigcup\limits_{i=1}^{n}I_{j_{i}} \supseteq I_{j_{3}},\end{equation*} 
and by Lemma~\ref{l:jk}  
applied to $g^{2}$ we obtain a compact interval 
$A_{j_{2}}\subseteq A_{j_1} \subseteq I_{j_{1}}$ such that
$g(A_{j_{2}}) = I_{j_{3}}$. 
Proceeding in a similar fashion, after $n$ consecutive applications of
Lemma~\ref{l:jk}, we obtain compact intervals 
\begin{equation*}A_{j_{n}}\subseteq A_{j_{n-1}}\subseteq \cdots\subseteq
A_{j_1}\subseteq I_{j_{1}},\end{equation*} 
such that 
\begin{equation}\label{e:gei}g^{i}(A_{j_{i}}) = I_{j_{i+1}},\quad
  i=1,\ldots,n,\end{equation}
in particular, $g^n(A_{j_n})=I_{j_{n+1}}=I_{j_1}$.   
Then, by Lemma~\ref{l:kk} there is a fixed point $a\in A_{j_n}$ for
$g^{n}$, hence $a$ is a periodic point for $g$ of period $n$.

We observe now that, for $n=1$ there are exactly $k$ different choices for $j_1=1,\ldots,k$ and, 
since the intervals $J_1,\ldots,J_k$ are mutually disjoint, the $k$ fixed points obtained, as explained 
before, are all different, hence $g$ has at least $k$ fixed points.

For $n\geq 2$, if
$j_1\neq j_i$ for all $i=2,\ldots,n$, 
then the periodic point $a$ for $g$
obtained before from the sequence $(j_i)_{i=1}^n$ has prime period
$n$. Indeed, for all $2\leq i\leq n$,
$a\in A_{j_i}\subseteq A_{j_{i-1}}$ and, by \eqref{e:gei} it follows that
$g^{i-1}(a)\in I_{j_i}$ hence, since
$I_{j_1}\cap I_{j_i}=\emptyset$ it follows that $a$ cannot have any period less
than $n$.

We consider now two different sequences $(j_{i})^{n}_{i=1}$,
$(l_{i})^{n}_{i=1}$ as before, hence there
exists $r$ $\in \{1, 2,..., n \}$ such that $j_{r} \neq l_{r}$. If the fixed
point $a$ for the two sequences of intervals $(I_{j_{i}})_{i=1}^n$ and 
$(I_{l_{i}})_{i=1}^n$, obtained as before, is the
same, then $(g^{i}(a))_{i=0}^n$, the orbit of $a$ under $g$, 
is the same for both sequences. But then,
$g^{r-1}(a)  \in  I_{j_{r}}\cap I_{l_{r}} = \emptyset $ hence we
have a contradiction. Therefore, for the $k(k-1)^{n-1}$ 
different sequences $(j_i)_{i=1}^n$ of length
$n$, formed with elements from the set $\{1,\ldots,k\}$ and
subject to the condition $j_1\neq j_i$ for
all $i=2,\ldots,n$, we have $k(k-1)^{n}$ different fixed points of prime
period $n$. 

(b) For any infinite sequence $(j_{i})^{\infty}_{i=1}$ with elements from
$\{1,2,...,k\}$, which is not eventually constant, 
by the same construction as above we get a sequence of nonempty 
compact intervals $(A_{j_i})_{i=1}^\infty$ subject to the properties
\begin{equation}\label{e:aiji}\cdots A_{j_{i+1}}\subseteq A_{j_i}\subseteq \cdots
\subseteq  A_{j_2}\subseteq A_{j_1}\subseteq I_{j_1},
\end{equation}
and
\begin{equation}\label{e:giaji} g^i(A_{j_i})=I_{j_{i+1}},\quad i\in\mathbb{N}.
\end{equation}
By the Finite Intersection Property we have 
\begin{equation*}A=\bigcap_{i=1}^\infty A_{j_i}\neq \emptyset,
\end{equation*} hence, any point $a\in A$ has the property that its orbit
$(g^i(a))_{i=1}^\infty$ makes a sequence that does not converge. The
sequence  $(g^i(a))_{i=1}^\infty$ does not converge since, for all
$i\in\mathbb{N}$ we have $g^i(a)\in I_{j_{i+1}}$, the sequence 
$(j_{i})^{\infty}_{i=1}$ with elements from
$\{1,2,...,k\}$ is not eventually constant, and the compact
intervals $I_1,\ldots,I_k$ are mutually disjoint. 

Let us consider two different sequences $(j_i)_{i=1}^\infty$ and 
$(l_i)_{i=1}^\infty$, formed
with elements from the set $\{1,\ldots,k\}$ and not eventually constant. As
before, to the sequence $(j_i)_{i=1}^\infty$ we associate the sequence of
nonempty compact intervals $(A_{j_i})_{i=1}^\infty$ subject to the properties
\eqref{e:aiji} and \eqref{e:giaji}, and let
$a\in \bigcap_{i=1}^\infty A_{j_i}$. Similarly, there is a sequence
$(B_{l_i})_{i=1}^\infty$ of nonempty compact intervals subject to the
properties
\begin{equation*}\label{e:biji}
\cdots B_{l_{i+1}}\subseteq B_{l_i}\subseteq \cdots
 \subseteq B_{l_2}\subseteq B_{l_1}\subseteq I_{j_1},
\end{equation*}
and
\begin{equation*}\label{e:gibji} g^i(B_{l_i})=I_{l_{i+1}},\quad i\in\mathbb{N},
\end{equation*}
 and let $b\in \bigcap_{i=1}^\infty B_{l_i}\neq\emptyset$. 
We claim that $a\neq b$. Indeed,
 there exists $r\in\mathbb{N}$ such that $j_r\neq l_r$, hence 
$g^r(a)\in I_{j_r}\cap I_{l_r}=\emptyset$, a contradiction.

In conclusion, there are as many real numbers $a$ for which the sequence
$(g^i(a))_{i=1}^\infty$ does not converge at least as many as sequences
$(j_i)_{i=1}^\infty$ formed
with elements from the set $\{1,\ldots,k\}$ and that are 
not eventually constant, and the latter set is uncountable.
\end{proof}

\section{The Dynamics of $M_f$}\label{s:dtontm}

Following \cite{HurleyMartin}, the \emph{Newton Class} is defined as the
collection of all real functions $f$ subject to the following conditions:
\begin{itemize}
\item[(nf1)] $f$ is of class $\mathcal{C}^2(\mathbb{R})$.
\item[(nf2)] If $f(x)=0$ then $f^\prime(x) \neq 0$. 
\item[(nf3)] If $f^\prime(x)=0$ then $f^{\prime\prime}(x) \neq 0$.
\end{itemize}

\begin{remarks}\label{r:snf} 
Let $f$ be a Newton map.

(a) Clearly, both $f$ and its derivative $f^\prime$ do
not have multiple roots.

(b) The roots of $f$, and similarly the roots of $f^\prime$, do not have finite
accumulation points. Indeed, if $(x_n)_{n\in\mathbb{N}}$ is a sequence of distinct
roots of $f$ that accumulates to some $x_0\in\mathbb{R}$ then, by the
Interlacing Property of the roots of $f$ and its derivative $f^\prime$, it
follows that there exists a sequence $(c_n)_{n\in\mathbb{N}}$ of distinct roots
of $f^\prime$ converging to $x_0$. Since both $f$ and $f^\prime$ are
continuous, it follows that $x_0$ is a root for both $f$ and $f^\prime$,
contradiction with the statement at item (a). A similar argument shows that the
roots of $f^\prime$ do not have finite accumulation points.
\end{remarks}

For a differentiable function $f$ on an open set $D\subseteq\mathbb{R}$, 
the \emph{Classical Newton's Approximation Function} 
$N_f$ is the function defined for all $x\in D$ such that
$f^\prime(x)\neq 0$ by
\begin{equation}\label{e:nef}N_{f}(x)=x-\frac{f(x)}{f^\prime(x)}.\end{equation}
Following \cite{AmatBusquierPlaza}, the 
\emph{Modified Newton's Approximation Function} $M_f$ is the
function, defined for all $x\in D$ such that $f^\prime(x)\neq 0$ and 
$N_f(x)\in D$, 
\begin{align}\label{e:mef}
 M_{f}(x) & =x-\frac{f(x)}{f^\prime(x)} -
  \frac{f(x-\frac{f(x)}{f^\prime(x)})}{f^\prime(x)}\\
\intertext{or, in terms of $N_{f}$,}
 M_f(x) & =N_{f}(x)-\frac {f(N_{f}(x))}{f^\prime(x)}. \label{e:mefnef}
\end{align}

The following lemma provides some information on the behavior of $N_f$ in the
neighbourhood of the critical points of a Newton map $f$, see Remark~1.3 in
\cite{HurleyMartin}. 

\begin{lemma}\label{l:lim}
Let $f$ be a Newton map and let
$c_{1}$ and $c_{2}$ be two consecutive roots of $f^\prime$ such that in the 
interval
  $(c_{1},c_{2})$ there is a unique root of $f$. Then 
\begin{equation*}
\lim\limits_{x \rightarrow c_{1}+} N_{f}(x)=-\lim\limits_{x \rightarrow
  c_{2}-}N_{f}(x)=\pm\infty,\end{equation*}
\end{lemma}

\begin{proof} Indeed, 
near $c_{1}+$ and $c_{2}-$, $f$ has opposite signs since it has a unique root
inside the interval $(c_1,c_2)$, 
while $f^\prime$ does not change its sign and goes to zero when
approaching both $c_{1}+$ and $c_{2}-$. Since the term $x$ is majorised by the
term $\frac{f(x)}{f^\prime(x)}$ near $c_{1}+$ and $c_{2}-$, we have the result. 
\end{proof}

Here is the main result that shows that the Modified Newton's Approximation
Function $M_f$ provides chaotic behaviour for a large class of Newton's
functions $f$.

\begin{theorem}\label{t:main}
Let the function $f\colon\mathbb{R}\rightarrow\mathbb{R}$ 
have the following properties:
\begin{itemize}
\item[(i)] $f$ is a Newton's function.
\item[(ii)] $\lim\limits_{x \to +\infty} f(x) 
= - \lim\limits_{x \to -\infty} f(x)=\pm \infty$.
\item[(iii)] $f$ has at least four real roots.
\end{itemize}
Then:

\emph{(a)} $M_{f}$ has periodic points of any prime period.

\emph{(b)} The set of all real numbers $a$ for which the sequence
$(M_f^n(a))_{n\in\mathbb{N}}$ does not converge is uncountable. 
\end{theorem}

\begin{proof} By Rolle's Theorem,
between any two consecutive roots of $f$ there exists a root of its derivative $f^\prime$ hence, 
by remarks~\ref{r:snf} and property (iii), we can choose four 
consecutive roots $r_{1}<r_{2}<r_{3}<r_{4}$ and three or four roots $c_{1}<c_{2}\leq 
c_2^\prime<c_{3}$ of $f^\prime$ 
in the intervals $(r_{i},r_{i+1})$ for $i=1,2,3$,
respectively, with a unique root of $f$
inside each of the intervals $(c_{1},c_{2}),(c_{2}^\prime,c_{3})$ 
and with no other root of $f^\prime$ there. 
With this choice, $f^\prime$ does not change its sign
on the intervals $(c_{1},c_{2})$ and $(c_{2}^\prime,c_{3})$. By using
Lemma~\ref{l:lim} and property (ii), we have
\begin{equation}\label{e:lixc}  \lim_{x \rightarrow c_{1}+}
  f(N_{f}(x))=-\lim_{x \rightarrow c_{2}-} f(N_{f}(x))=\pm
  \infty. \end{equation} 

In the following we show that
\begin{equation}\label{e:lmf} \lim_{x \rightarrow c_{1}+}
  M_{f}(x)=-\lim_{x \rightarrow c_{2}-}M_{f}(x)=\pm \infty. \end{equation}
Indeed, when $x\rightarrow c_1+$, 
the first equality will follow if we show that
\begin{equation*}\lim_{x \rightarrow c_{1}{+}}=
\frac{-f(x)-f(N_{f}(x))}{f^\prime(x)}=\pm\infty,
\end{equation*} and, taking into account that $\lim\limits_{x\rightarrow
  c_1+}f(x)=f(c_1)\in\mathbb{R}$, we observe that the latter will follow if we
prove that
\begin{equation*} 
\lim_{x \rightarrow c_{1}+}\frac{-f(N_{f}(x))}{f^\prime(x)}=\pm\infty,
\end{equation*} which actually follows from \eqref{e:lixc} 
and the fact that $f^\prime$ does not change its sign in the interval
$(c_{1},c_{2})$. Hence, the first equality in \eqref{e:lmf} is proven.
The fact that $\lim\limits_{x \rightarrow c_{2}{-}}M_{f}(x)=\pm \infty$ is proven
similarly, the only thing that remains to be shown is that the two limits in
\eqref{e:lmf} have different signs, which actually is a consequence of
\eqref{e:lixc} and the fact that $f^\prime$ has constant sign on $(c_1,c_2)$. 
Therefore, \eqref{e:lmf} is proven.

Similarly, we have that 
\begin{equation*} \lim_{x \rightarrow c_{2}^\prime+}
  M_{f}(x)=-\lim_{x \rightarrow c_{3}-}M_{f}(x)=\pm \infty. \end{equation*} 

Therefore, 
there exists $\varepsilon$ sufficiently small such that, letting
$I_{1}=[c_{1}+\varepsilon,c_{2}-\varepsilon]$  and
$I_2=[c_2^\prime+\varepsilon,c_3-\varepsilon]$ we observe that $I_1\cap I_2=\emptyset$
and that
$M_{f}(I_{j})\supseteq [c_1,c_3]$, for $j=1,2$, hence $M_f(I_1)$ and
$M_f(I_{2})$ contain $I_{1}\cup I_{2}$. Finally,
Lemma~\ref{l:jkn} is now applicable with $k=2$, which finishes the proof. 
\end{proof}

\begin{remark}\label{r:dnm}
In order to reduce the chaotic behaviour and improve numerical parameters of
approximation for lower order polynomials, 
in \cite{AmatBusquierMagrenan} and \cite{MagrenanGutierez}
damped Newton's methods have been considered. More precisely, letting
$\lambda$ be the damping parameter, one defines $N_{\lambda,f}$ and
  $M_{\lambda, f}$ as follows:
\begin{equation}\label{e:dnef}N_{\lambda,f}(x)=x-\lambda\frac{f(x)}{f^\prime(x)},
\end{equation}
and
\begin{equation}\label{e:dmef}
 M_{\lambda,f}(x) =N_{\lambda,f}(x)-\lambda\frac {f(N_{\lambda,f}(x))}{f^\prime(x)}.
\end{equation}
It is easy to observe, by inspection, that Lemma~\ref{l:lim} and
Theorem~\ref{t:main} remain true if $N_{\lambda,f}$ and
$M_{\lambda,f}$ replace $N_{f}$ and, respectively, $M_f$, for arbitrary
damping parameter $\lambda>0$, hence the chaotic behaviour
characterised by existence of periodic points of any prime period, as well as
the uncountability of the set of points of divergence of iteration of $M_f$, 
remain unaltered by damping, for the class of functions considered in
Theorem~\ref{t:main}.
\end{remark}

%\section{An Estimation of the Bowen's Topological Entropy}

\section{The Scaling Theorem}\label{s:tst}

In this section we observe that the Scaling
Theorem, cf.\ Theorem~1 in \cite{AmatBusquierPlaza},
which says that $M_{f}$ is stable
under affine conjugation, remains true if we replace the analyticity 
condition on the function $f$ is with its differentiability. 

\begin{theorem}[The Scaling Theorem]\label{t:scaling}
Let $f$ be a differentiable function on $\mathbb{R}$. 
Let $T(x)=ax+b$ be an affine map with 
$a\neq 0$.
Then, for all $x\in\mathbb{R}$ with $f^\prime(x)\neq 0$, we
have 
\begin{equation*}\bigl(T\circ M_{f\circ T}\circ
    T^{-1}\bigr)(x)=M_{f}(x).\end{equation*}
\end{theorem}

We first prove a lemma saying that the transformation $N_f$ satisfies a
similar property of stability under linear conjugation.

\begin{lemma}\label{l:tec} 
Under the assumptions as in Theorem~\ref{t:scaling}, 
for all $x\in\mathbb{R}$ with $f^\prime(x)\neq 0$, we have 
 \begin{equation*}\bigl(T\circ N_{f\circ T}\circ
    T^{-1}\bigr)(x)=N_{f}(x).\end{equation*} 
\end{lemma}

\begin{proof} Clearly, for all $y\in\mathbb{R}$ we have
$(f\circ T)'(y)=af^\prime(ay+b)$ hence, if $f^\prime(ay+b)\neq 0$ we have
\begin{equation}\label{e:neft}
N_{f\circ
      T}(y)=y-\frac{f(ay+b)}{af^\prime(ay+b)},\end{equation}
and then
\begin{align*}\bigl(T\circ N_{f\circ
      T}\bigr)(y) & = ay-a\frac{f(ay+b)}{af^\prime(ay+b)}+b\\ 
& = ay+b-\frac{f(ay+b)}{f^\prime(ay+b)} \\
& = N_f(ay+b)=(N_f\circ T)(y).
\end{align*} 
Since $T^{-1}(x)=\frac{x-b}{a}$, letting $x=ay+b$ we get 
\begin{equation*}\bigl(T\circ N_{f\circ
    T}\circ T^{-1}\bigr)(x)=N_{f}(x).\qedhere\end{equation*} 
\end{proof}

\begin{proof}[Proof of Theorem~\ref{e:neft}] By \eqref{e:mef},
for arbitrary $y\in\mathbb{R}$ with
$f^\prime(ay+b)\neq 0$, from \eqref{e:neft} we obtain
\begin{equation*} \bigl(T\circ M_{f\circ T}\bigr)(y)
=aN_{f\circ T}(y)+b-\frac{f(aN_{f\circ
      T}(y)+b)}{f^\prime(ay+b)},\end{equation*} and then, letting $x=ay+b$, 
\begin{align*} \bigl(T\circ M_{f\circ T}\circ T^{-1}\bigr)(x)
 & =\bigl(T\circ N_{f\circ T}\circ
  T^{-1}\bigr)(x)-\frac{f(\bigl(T\circ N_{f\circ T}\circ
    T^{-1}\bigr)(x))}{f^\prime(a(\frac{x-b}{a})+b)}\\
\intertext{whence, applying Lemma~\ref{l:tec},}
 & =N_{f}(x)-\frac
  {f(N_{f}(x))}{f^\prime(x)}=M_{f}(x).\qedhere\end{align*} 
\end{proof}

\begin{remark}\label{r:dst} The Scaling Theorem remains true, with almost
  exactly the same proof, for the damped Newton's function $M_{\lambda,f}$, 
see \eqref{e:dmef}, for arbitrary damping parameter $\lambda\neq 0$, more
precisely, we have
\begin{equation*}\bigl(T\circ M_{\lambda, f\circ T}\circ
    T^{-1}\bigr)(x)=M_{\lambda,f}(x),
\end{equation*} for all $x\in\mathbb{R}$ such that $f^\prime(x)\neq 0$.
For the case of the damped Newton's function $N_{\lambda,f}$, see
\eqref{e:dnef}, the corresponding generalisation of Lemma~\ref{l:tec},
\begin{equation*}\bigl(T\circ N_{\lambda,f\circ T}\circ
    T^{-1}\bigr)(x)=N_{\lambda,f}(x),\end{equation*} 
 follows from the
  proof of Theorem~2.1 in \cite{MagrenanGutierez}; although it was stated for
  analytic functions $f$, the assumption was not used in that proof.
\end{remark}

\end{document}